\pdfoutput=1
\RequirePackage{ifpdf}
\ifpdf 
\documentclass[pdftex]{sigma}
\else
\documentclass{sigma}
\fi

\newtheorem*{Theorem4.5}{Theorem \ref{n=2Cor}}
\newtheorem{Theorem}{Theorem}[section]
\newtheorem{Corollary}[Theorem]{Corollary}
\newtheorem{Lemma}[Theorem]{Lemma}
\newtheorem{Proposition}[Theorem]{Proposition}
 { \theoremstyle{definition}
\newtheorem{Definition}[Theorem]{Definition}
\newtheorem{Example}[Theorem]{Example}
\newtheorem{Remark}[Theorem]{Remark} }

\numberwithin{equation}{section}

\usepackage{tikz-cd}
\usetikzlibrary{matrix,arrows,decorations.pathmorphing}

\def\ed{\mathrm{d}}
\def\I{\mathcal I}
\def\R{\mathbb R}

\def\x{{\bf x}}
\def\f{{\bf f}}
\def\u{{\bf u}}
\def\v{{\bf v}}
\def\g{{\bf g}}
\def\y{{\bf y}}
\def\rank{\operatorname{rank}}
\def\spn{\operatorname{span}}
\def\diag{\operatorname{diag}}
\def\W{\wedge}
\def\ct{\mathcal{C}}
\def\bs{\boldsymbol}
\def\bsomega{{\bs{\omega}}}
\def\bseta{{\bs{\eta}}}

\begin{document}
\allowdisplaybreaks

\newcommand{\arXivNumber}{1902.00598}

\renewcommand{\PaperNumber}{063}

\FirstPageHeading

\ShortArticleName{Dynamic Equivalence of Control Systems and Infinite Permutation Matrices}

\ArticleName{Dynamic Equivalence of Control Systems\\ and Infinite Permutation Matrices}

\Author{Jeanne N.~CLELLAND~$^\dag$, Yuhao HU~$^\dag$ and Matthew W.~STACKPOLE~$^\ddag$}

\AuthorNameForHeading{J.N.~Clelland, Y.~Hu and M.W.~Stackpole}

\Address{$^\dag$~Department of Mathematics, 395 UCB, University of Colorado, Boulder, CO 80309-0395, USA}
\EmailD{\href{mailto:Jeanne.Clelland@colorado.edu}{Jeanne.Clelland@colorado.edu}, \href{mailto:Yuhao.Hu@colorado.edu}{Yuhao.Hu@colorado.edu}}

\Address{$^\ddag$~Maxar Technologies, 1300 W.~120th Ave, Westminster, CO 80234, USA}
\EmailD{\href{mailto:Matt.Stackpole@maxar.com}{Matt.Stackpole@maxar.com}}

\ArticleDates{Received February 05, 2019, in final form August 20, 2019; Published online August 26, 2019}

\Abstract{To each dynamic equivalence of two control systems is associated an infinite permutation matrix. We investigate how such matrices are related to the existence of dynamic equivalences.}


\Classification{34H05; 58A15; 58A17}

\section{Introduction}

A \emph{control system} is an underdetermined ODE system of the form
\begin{gather*}
\dot \x = \f(t,\x,\u),
\end{gather*}
where $\x = (x_i)$ are called the \emph{state variables} and $\u = (u_\alpha)$ the \emph{control variables}. The meaning of ``control'' is clear: Under suitable regularity conditions, specifying a~control function $\u(t)$ and an initial state $\x(t_0)$ uniquely determines a local ``trajectory'' $\x(t)$ that satisfies the
ODE system and the initial state. When $\f$ does not explicitly depend on $t$, the control system is said to be \emph{autonomous}, which we will assume throughout this paper.

Let $\dot \x = \f(\x,\u)$ and $\dot\y = \g(\y,\v)$ be two control systems. Suppose that there exist mappings $\phi = \phi\big(\x,\u,\dot\u,\ddot\u,\dots ,\u^{(p)}\big)$ and $\psi=\psi\big(\y,\v,\dot\v,\ddot\v,\dots ,\v^{(q)}\big)$ such that,
\begin{itemize}\itemsep=0pt
\item for any solution $(\x(t), \u(t))$ of the first system, the function $(\y,\v) = \phi\big(\x,\u,\dot\u,\ddot\u,\dots ,\u^{(p)}\big)$ is a solution of the second;
\item for any solution $(\y(t), \v(t))$ of the second system, the function $(\x,\u) = \psi\big(\y,\v,\dot\v,\ddot\v,\dots ,\allowbreak \v^{(q)}\big)$ is a solution of the first;
\item moreover, applying $\phi$ and $\psi$ successively to a solution $(\x(t),\u(t))$ of the first system yields the same solution $(\x(t), \u(t))$, and, similarly, applying $\psi$ and $\phi$ successively to a solution $(\y(t), \v(t))$ of the second system yields the same solution $(\y(t), \v(t))$.
\end{itemize}
If all these conditions are satisfied, we say that the pair of maps $(\phi,\psi)$ establishes a \emph{dynamic equivalence} between the two control systems. Intuitively, a dynamic equivalence provides a~one-to-one correspondence between the solutions of one control system with those of the other.

Fixing a dynamic equivalence $(\phi,\psi)$, one can always find the smallest $p,q\ge 0$ so that $\phi = \phi\big(\x,\u,\dot\u,\ddot\u,\dots ,\u^{(p)}\big)$ and $\psi=\psi\big(\y,\v,\dot\v,\ddot\v,\dots ,\v^{(q)}\big)$. We call such a pair $(p,q)$ the \emph{height} of the corresponding dynamic equivalence. A dynamic equivalence with height $(0,0)$ is known as a \emph{static $($feedback$)$ equivalence}, in which case $\phi$, $\psi$ are inverses of each other as diffeomorphisms.

An immediate question is: \emph{How much more general is the notion of dynamic equivalence than that of static equivalence?} Classical results (see~\cite{MMR03,pomet95}) suggest that the answer depends on the number of control variables. In particular, a dynamic equivalence between two control systems with a single control variable is necessarily static. It is also well known that the number of control variables is invariant under a dynamic equivalence. However, in the cases of~2 or more controls, a precise answer to the question above remains largely unknown.

In \cite{stackpole}, the author considered all \emph{control-affine} systems with 3 states and 2 controls, proving that three statically non-equivalent systems are pairwise dynamically equivalent at height~$(1,1)$. In addition, he introduced a new method of studying dynamic equivalences of two control systems. He found that to each dynamic equivalence is associated an infinite permutation matrix. Intuitively, such a matrix tells us how the `generating $1$-forms' of certain prolongations of the two control systems (viewed as Pfaffian systems), when chosen appropriately, relate under a~dynamic equivalence.

In the current work, we present further properties of dynamic equivalences that can be derived using the associated infinite permutation matrices. First we prove that there is a \emph{rank matrix} (Definition~\ref{rankMatrixDef}) associated to a dynamic equivalence, which has a more `invariant' nature than an associated infinite permutation matrix (Proposition~\ref{rankProp}). Then we prove several inequalities and equalities (Propositions~\ref{rankEquality} and~\ref{rankInequality2}) satisfied by the rank matrix. Using these results, we prove an inequality satisfied by the height $(p,q)$ of a dynamic equivalence (Theorem~\ref{heightTheorem}). In particular,
this inequality implies the

\begin{Theorem4.5}The height $(p,q)$ of a dynamic equivalence between two control systems with $n_1$ and $n_2$ states, respectively, and $2$ controls must satisfy
$n_1+p = n_2+q$.
\end{Theorem4.5}

\section{Control systems and dynamic equivalence}
\subsection{Control systems}

\begin{Definition} \label{DefDynSys}
A \emph{control system with $n$ states and $m$ controls} is an underdetermined ODE system
\begin{gather}\label{ctrlSys}
\dot\x = \f(\x,\u),
\end{gather}
where \begin{gather*}\x = (x_i)\in \R^n, \qquad \u = (u_\alpha)\in \R^m,\end{gather*}
and $\f = (f_i)\colon \R^{n+m}\rightarrow \R^n$ is a smooth function satisfying $\rank\big(\frac{\partial f_i}{\partial u_\alpha}\big) = m$ on some open domain\footnote{Since our study is local, we henceforth assume that such a domain is the entire $\R^{n+m}$.} in $\R^{n+m}$. Here, $x_i$ are called the \emph{state variables}, $u_\alpha$ the \emph{control variables}.
\end{Definition}

For a control system, there is an equivalent geometric characterization. Let $\mathcal{D}\subset \R^n$ be the domain of the state variables $\x = (x_i)$. The admissible $t$-derivatives of $x_i$, as imposed by the control system, are given by specifying a submanifold $\Sigma \subset T\mathcal{D}$ that submerses onto $\mathcal{D}$ with rank-$m$ fibers. Each fiber is precisely parametrized by the control variables $\u= (u_\alpha)$. The submanifold $\Sigma$ induces an embedding
\begin{gather*}
\iota\colon \ \R\times \Sigma \hookrightarrow \R \times T\mathcal{D},
\end{gather*}
which is the identity in the $\R$-factor (with coordinate $t$) and $t$-independent in the $\Sigma$-factor. In coordinates, this embedding may be written as
\begin{gather*}
\iota(t, \x,\u) = (t, \x,\f(\x, \u)),
\end{gather*}
which satisfies
\begin{gather*}
\iota^*(\ed x_i - \dot x_i \ed t) = \ed x_i - f_i(\x,\u)\ed t.
\end{gather*}
In other words, the system \eqref{ctrlSys} corresponds to the Pfaffian system $(M,\ct_M)$, where $M:= \R\times \Sigma$, and $\ct_M$ is the restriction to $M$ of the standard contact system $\ct = \langle\ed x_i - \dot x_i \ed t\rangle_{i = 1}^n$ on the jet bundle $J^1(\R,\mathcal{D})\cong \R\times T\mathcal{D}$. Conversely, let $\Sigma\subset T\mathcal{D}$ be a submanifold that submerses onto a~domain $\mathcal{D}\subset \R^n$ with rank-$m$ fibers. The Pfaffian system $(M,\ct_M)$ corresponds to a~control system with $n$ states and $m$ controls.

\subsection{Prolongations of a control system}

\begin{Definition} Let $(M,\ct_M)$ be a control system with $n$ states and $m$ controls. Let $(t,\x,\u)$ be coordinates on~$M$. Suppose that $\ct_M$ is generated by $\ed x_i - f_i(\x,\u)\ed t$ $(i = 1,\dots ,n)$. By definition, the \emph{first total prolongation}\footnote{Geometrically, $M^{(1)}$ is known as the \emph{space of integral line elements} of $(M,\I)$. See \cite{BCG}. In particular, one can show that this definition of $\big(M^{(1)}, \ct^{(1)}\big)$ is independent of the choice of coordinates on $M$.} of $(M,\ct_M)$ is the Pfaffian system $\big(M^{(1)}, \ct^{(1)}\big)$, where $M^{(1)} = M\times\R^m$ with the coordinates $\big(t,\x,\u,\u^{(1)}\big)$; $\ct^{(1)}$ is the Pfaffian system generated by
\begin{gather*}
\ed x_i - f_i(\x,\u)\ed t, \qquad \ed u_\alpha - u_\alpha^{(1)}\ed t, \qquad i = 1,\dots ,n, \qquad \alpha = 1,\dots ,m.
\end{gather*}
\end{Definition}

Let $k$ be a positive integer. One can start from a control system $(M,\ct_M)$ with $n$ states and $m$ controls and generate total prolongations successively for $k$ times. The result is called the \emph{$k$-th total prolongation of $(M,\ct)$}, denoted as $\big(M^{(k)}, \ct^{(k)}\big)$, where $M^{(k)}$ has the coordinates $\big(t,\x,\u, \u^{(1)}, \dots , \u^{(k)}\big)$ and $\ct^{(k)}$ is generated by the $1$-forms
\begin{gather*}
 \ed x_i - f_i(\x,\u)\ed t, \qquad \ed u_\alpha - u_\alpha^{(1)}\ed t, \qquad \ed u_\alpha^{(\ell)} - u_{\alpha}^{(\ell+1)}\ed t, \\
 i = 1,\dots ,n,\qquad \alpha = 1,\dots ,m,\qquad \ell = 1,\dots ,k-1.
\end{gather*}

When $k= 0$, we simply let $\big(M^{(0)}, \ct^{(0)}\big)$ denote $(M,\ct_M)$.

It is clear that $\big(M^{(k)},\ct^{(k)}\big)$ is a control system with $n+km$ states and $m$ controls.

\subsection{Dynamic equivalence}

Given two control systems, it is natural to regard them as equivalent if one can establish a~one-to-one correspondence between their solutions.

Of course, two control systems $\dot\x = \f(\x,\u)$ and $\dot\y = \g(\y,\v)$ are equivalent in the sense above when they can be transformed into each other by a change of variables of the form $\y = \phi(\x)$, $\v = \xi(\x,\u)$ and $\x = \phi^{-1}(\y)$, $\u = \rho(\y,\v)$. This notion of equivalence is called \emph{static equivalence}, which, in particular, requires the two equivalent control systems to have same number of states and the same number of controls. However, it is possible for two systems with different numbers of states to have a one-to-one correspondence between their solutions, as is indicated by the following standard property of jet bundles.

\begin{Proposition}\label{PropProlong}Let $(M,\ct_M)$ be a control system. Let $\pi\colon M^{(k)}\rightarrow M$ be the canonical projection from its $k$-th total prolongation. Any integral curve $\tau\colon \R\rightarrow M$ of $(M,\ct_M)$ has a unique lifting $\tau^{(k)}\colon \R\rightarrow M^{(k)}$ $($i.e., satisfying $\pi\circ\tau^{(k)} = \tau)$ to an integral curve of $\big(M^{(k)}, \ct^{(k)}\big)$. In addition, for each integral curve $\sigma\colon \R\rightarrow M^{(k)}$ of $\big(M^{(k)}, \ct^{(k)}\big)$, its projection $\pi\circ\sigma$ is an integral curve of~$(M,\ct_M)$.
\end{Proposition}

In other words, given two control systems, a one-to-one correspondence between their solutions may involve differentiation. This motivates the following notion of equivalence.

\begin{Definition}\label{DynEqDef} Two control systems $(M, \ct_M)$ and $(N,\ct_N)$ are said to be \emph{dynamically equivalent} if there exist integers $p, q \geq 0$ and submersions $\Phi\colon M^{(p)}\rightarrow N$ and $\Psi\colon N^{(q)}\rightarrow M$ that satisfy\footnote{A more careful definition would set the domains of $\Phi$ and $\Psi$ to be open subsets of $M^{(p)}$ and $N^{(q)}$, respectively. See, for example, \cite{pomet95}. Since our results are local, for the economy of notations, we will be content with
the definition presented here.}
\begin{enumerate}\itemsep=0pt
\item[(i)] $\Phi$, $\Psi$ preserve the $t$-variable and are $t$-independent in the state and control components;
\item[(ii)] $\Phi$ cannot factor through any $M^{(k)}$ for $k<p$; $\Psi$ cannot factor through any $N^{(\ell)}$ for $\ell<q$;
\item[(iii)] for each integral curve $\tau\colon \R\rightarrow M$ of $(M,\ct_M)$, $\Phi\circ\tau^{(p)}$ is an integral curve of $(N,\ct_N)$; for each integral curve $\sigma\colon \R\rightarrow N$ of $(N,\ct_N)$, $\Psi\circ\sigma^{(q)}$ is an integral curve of $(M,\ct_M)$;
\item[(iv)] letting $\tau$ and $\sigma$ be as in (iii), we have
\begin{gather*}
\tau =\Psi\circ\big(\Phi\circ \tau^{(p)}\big)^{(q)}, \qquad \sigma =\Phi\circ\big(\Psi\circ \sigma^{(q)}\big)^{(p)}.
\end{gather*}
\end{enumerate}
\end{Definition}

For the convenience of the reader, we present the commutative diagram:
\begin{equation*}
\begin{tikzcd}[column sep = 32, row sep = 30]
& M^{(p)} \arrow[dr,"\phantom{aa}\Phi", sloped]\arrow[d,"\pi", swap] & N^{(q)} \arrow[dl, "\Psi\phantom{aa}",swap, sloped]\arrow [d,"\pi"]&\\
\R \arrow[ur,"\tau^{(p)}"]\arrow[r,"\tau"]& M & N & \R. \arrow[ul,"\sigma^{(q)}", swap]\arrow[l,"\sigma", swap]
 \end{tikzcd}
\end{equation*}

\begin{Remark} We observe the following:
\begin{enumerate}\itemsep=0pt
\item[(a)]It is easy to verify that Definition \ref{DynEqDef} defines an equivalence relation.
\item[(b)] By this definition, a control system $(M,\ct_M)$ is dynamically equivalent to each of its total prolongations $\big(M^{(k)},\ct^{(k)}\big)$.
\item[(c)] A dynamic equivalence with $p = q = 0$ is a static equivalence. To see this, let $(t,\x,\u)$ and $(t,\y,\v)$ be coordinates on~$M$ and~$N$, respectively. Represent $\Phi$ in local coordinates as $(t,\y,\v) = \big(t,\phi^s(\x, \u), \phi^c(\x,\u)\big)$. (Here the superscripts `$s$' and `$c$' of $\phi$ stand for `state' and `control', respectively.) Since $\Phi$ maps integral curves of $(M,\ct_M)$ to integral curves of $(N,\ct_N)$, it is necessary that, for each $\ed y_i - g_i(\y,\v)\ed t\in \ct_N$, its pull-back
\begin{gather*}\Phi^*(\ed y_i - g_i(\y,\v)\ed t) = \ed \big(\phi^s_i(\x,\u)\big) - g_i\big(\phi^s(\x,\u), \phi^c(\x,\u)\big)\ed t\end{gather*}
is contained in $\ct_M$. It follows that $\phi^s(\x,\u)$ is independent of $\u$. A similar argument applies to~$\Psi$. Finally, Condition~(iv) in Definition~\ref{DynEqDef}
implies that $\Phi$ and $\Psi$ are inverses of each other.
\item[(d)] This definition of dynamic equivalence corresponds to the notion of \emph{endogenous transformation} in the control literature (see~\cite{MMR03}). In broader contexts, it is related to the notion of \emph{Lie--B\"acklund equivalences} (see \cite{AI79,FLMR99,Levine11}), that of \emph{$\mathcal{C}$-transformations} (see \cite{Chetverikov17,KLVjet}) and equivalences between differential algebras (see~\cite{Jakubczyk94}).
\end{enumerate}
\end{Remark}

\begin{Definition}We call the pair of integers $(p,q)$ in Definition~\ref{DynEqDef} the \emph{height} of the corresponding dynamic equivalence.
\end{Definition}

\begin{Remark} Keeping the notations from the above, $p$ is the highest derivative of $\u$ that $\Phi$ depends on, and $q$ is the highest derivative of $\v$ that $\Psi$ depends on. On the other hand, one may find the highest derivative of each $v_\beta$ that $\Psi$ depends on and sum over~$\beta$. Of course, this depends on the choice of coordinates on $N$. Such a sum is related to the notion of \emph{differential weight} (of~$\Psi$) defined in~\cite{Nicolau17}, where its relation with differential flatness is investigated.
\end{Remark}

Given two dynamically equivalent control systems $(M,\ct_M)$ and $(N,\ct_N)$, if needed, one could always apply a \emph{partial prolongation} (for details, see~\cite{stackpole}) to one of them such that the resulting systems have the same number of states and are still dynamically equivalent. When this is achieved, the proposition below
would become applicable.

 \begin{Proposition}\label{porq=0Prop} Let $(M,\ct_M)$ and $(N,\ct_N)$ be control systems with the same number of states. The height $(p,q)$ of a dynamic equivalence
between them must satisfy either $p = q = 0$ or $p,q > 0$.
 \end{Proposition}
\begin{proof}Suppose that the following commutative diagram represents a dynamic equivalence of height $(p,0)$ $(p>0)$ between $(M,\ct_M)$ and $(N,\ct_N)$:
\begin{equation*}
\begin{tikzcd}[column sep=small]
 \big(M^{(p)},{\ct^{(p)}}\big) \arrow[d,"\displaystyle\pi",swap]\arrow[rd,"\Phi"] & \\
(M,\ct_M) &(N,\ct_N).\arrow[l,"\Psi",swap]
 \end{tikzcd}
\end{equation*}
By the assumption, $\rank(\ct_M) = \rank(\ct_N)$. Since
\begin{gather*}
\pi^*\ct_M = (\Psi\circ\Phi)^*\ct_M = \Phi^*(\Psi^*\ct_M)\subseteq \Phi^*\ct_N,
\end{gather*}
and since $\pi$, $\Phi$, $\Psi$ are all submersions, it is necessary that $\langle\pi^*\ct_M\rangle = \langle\Phi^*\ct_N\rangle$. Now, $\Phi$ is constant along the Cauchy characteristics of $\Phi^*\ct_N$, since the Cartan system (see~\cite{BCG}) of $\ct_N$ generates the entire cotangent bundle of~$N$. On the other hand, the Cauchy characteristics
of $\pi^*\ct_M$ are precisely the fibres of $\pi$. This proves that $\Phi$ factors through $M$, a contradiction to the choice of~$p$. The case when $p = 0$, $q> 0$ is similar.
\end{proof}

\section[Infinite permutation matrices associated to a dynamic equivalence]{Infinite permutation matrices associated\\ to a dynamic equivalence}

In this section, we assume $p, q>0$ unless otherwise noted.

Given a control system $(M,\ct_M)$, one automatically obtains a system of projections
\begin{gather*}
\pi_{k,j}\colon \ M^{(k)}\rightarrow M^{(j)}, \qquad k\ge j.
\end{gather*}
The inverse limit of this projective system is denoted as
\begin{gather*}
M^{(\infty)}: = \varprojlim_{k} M^{(k)}.
\end{gather*}
Let $\pi_{k}\colon M^{(\infty)}\rightarrow M^{(k)}$ be the canonical projections. Since $\pi_{k,j}^{*}\ct^{(j)}\subseteq \ct^{(k)}$ for all $k\ge j\ge 0$, one can define $\ct^{(\infty)}$ to be the differential system generated by $\bigcup_{k\ge 0} \pi_k^{*}\ct^{(k)}$.

In coordinates, if $(M,\ct_M)$ corresponds to the system $\dot \x = \f(\x,\u)$, then $M^{(\infty)}$ has the standard coordinates $\big(t,\x,\u,\u^{(1)},\dots \big)$; $\ct^{(\infty)}$ is generated by the $1$-forms
\begin{gather*}
 \ed x_i - f_i(\x,\u)\ed t, \qquad \ed u_\alpha - u_\alpha^{(1)}\ed t, \qquad \ed u_\alpha^{(\ell)} - u_{\alpha}^{(\ell+1)}\ed t, \\
 i = 1,\dots ,n,\qquad \alpha = 1,\dots ,m, \qquad \ell \ge 1.
\end{gather*}
The pair $\big(M^{(\infty)}, \ct^{(\infty)}\big)$ is called the \emph{infinite prolongation} of $(M,\ct_M)$.

Now suppose that $(M,\ct_M)$ $(\dot\x = \f(\x,\u))$ and $(N,\ct_N)$ $(\dot\y = \g(\y,\v))$ are two control systems with $n_1$, $n_2$ states and $m_1$, $m_2$ controls, respectively. (Note that, at this point, we do not assume that either the number of states or the number of controls is the same for both systems). Furthermore, suppose that a dynamic equivalence of height $(p,q)$ between $(M, \ct_M)$ and $(N,\ct_N)$ is given by
\begin{gather*}
\Phi\colon \ M^{(p)}\rightarrow N, \qquad \Psi\colon \ N^{(q)}\rightarrow M.
\end{gather*}
It can be shown that $\Phi$ and $\Psi$ induce, respectively, maps
\begin{gather*}
\Phi^{(\infty)}\colon \ M^{(\infty)}\rightarrow N^{(\infty)}, \qquad \Psi^{(\infty)}\colon \ N^{(\infty)}\rightarrow M^{(\infty)},
\end{gather*}
which satisfy
\begin{gather*}
\Phi^{(\infty)}\circ \Psi^{(\infty)} = {\rm Id}, \qquad \Psi^{(\infty)} \circ\Phi^{(\infty)} = {\rm Id}.
\end{gather*}
Moreover, if we let
\begin{gather*}
\bsomega^{0} = \ed \x - \f(\x,\u)\ed t, \qquad \bsomega^{k} = \ed \u^{(k-1)} - \u^{(k)}\ed t, \qquad k\ge 1,\\
\bseta^{0} = \ed \y - \g(\y,\v)\ed t, \qquad \bseta^{k} = \ed \v^{(k-1)} - \v^{(k)}\ed t, \qquad k\ge 1,
\end{gather*}
then there exist matrices $A^i_j$, $B^i_j$ $(i,j \ge 0)$ satisfying\footnote{Here we have assumed $p,q>0$, otherwise, these properties may not hold. For instance, consider the standard dynamic equivalence between an arbitrary $(M,\ct_M)$ and $\big(M^{(k)}, \ct^{(k)}\big)$.}
\begin{enumerate}\itemsep=0pt
\item[P1.] \emph{for $k\ge 0$, we have}
\begin{gather}
{\Phi^{(\infty)}}^*\bseta^k = A^k_0\bsomega^0+ A^k_1\bsomega^1 + \cdots + A^k_{p+k}\bsomega^{p+k}, \label{om2eta}\\
{\Psi^{(\infty)}}^*\bsomega^k = B^k_0\bseta^0+ B^k_1\bseta^1 + \cdots + B^k_{q+k}\bseta^{q+k},\label{eta2om}
\end{gather}
where $A^{k}_{p+k}, B^{k}_{q+k}\ne 0$ for all $k\ge 0$.
\item[P2.] \emph{All $A_{p+k}^k$ are equal for $k\ge 1$; similarly for $B_{q+k}^k$. Hence, we denote}
\begin{gather*}A_{\infty} := A_{p+k}^k, \qquad B_{\infty} := B_{q+k}^k,\qquad k\ge 1.\end{gather*}
\item[P3.] $\rank \big(A^{0}_p\big) = \rank(A_{\infty})>0$, $\rank\big(B^{0}_q\big) = \rank(B_\infty)>0$.
\end{enumerate}

In particular, P1 follows from the definition of a dynamic equivalence of height $(p,q)$; P2 can be seen by taking the exterior derivative of~\eqref{om2eta} on both sides for $k\ge 1$ then reducing modulo $\bsomega^0,\dots ,\bsomega^{p+k}$ and similarly for~\eqref{eta2om}. For more details of these proofs, see~\cite{stackpole}.\footnote{The proofs in~\cite{stackpole} were made under the assumptions: $n_1 = n_2$ and $m_1 = m_2$, but they can be easily modified to work for the situation being considered here.}

To prove P3, note that there exists an $n_2\times m_2$ matrix $F$ such that
\begin{align*}
\ed\bseta^0 &\equiv F\bseta^1 \W \ed t \mod \bseta^0\\
&\equiv FA^1_{p+1}\bsomega^{p+1} \W \ed t\mod \bs \bsomega^0, \bsomega^1,\dots ,\bsomega^p.
\end{align*}
On the other hand,
\begin{align*}
\ed\bseta^0& = \ed\big(A^0_0\bsomega^0+\cdots + A^0_p\bsomega^p\big)\\
& \equiv -A^0_p\bsomega^{p+1}\W \ed t\mod \bsomega^0,\bsomega^1,\dots ,\bsomega^p.
\end{align*}
Consequently,
\begin{gather*}
FA^1_{p+1} = -A^0_p.
\end{gather*}
By Definition~\ref{DefDynSys}, $F$ has full rank. Therefore, $\rank\big(A^1_{p+1}\big)= \rank\big(A^0_p\big)$. The case for $B^i_j$ is similar. Property P3 follows.

\subsection{Two classical theorems}

Concerning dynamic equivalences between two control systems, the following two classical theorems are of fundamental importance.

\begin{Theorem}[\cite{MMR03, pomet95}]\label{sameSTheorem} Two dynamically equivalent control systems must have the same number of control variables.
\end{Theorem}

\begin{Theorem}[\cite{cartan1914courbes,Sluis}] \label{s=1Theorem}A dynamic equivalence between two control systems with~$n$ states and single control can only have height $(p,q) = (0,0)$ $($i.e., the equivalence is static$)$.
\end{Theorem}

\begin{Remark} The original result of Cartan takes a different form. In particular, the notion of equivalence considered in~\cite{cartan1914courbes} is \emph{absolute equivalence}. On the other hand, it is a consequence of Corollary~26 in~\cite{Sluis} that two control systems are absolutely equivalent (in an appropriate sense as $t$-independence of maps is assumed) if and only if they are dynamically equivalent. Thus, Theorem~\ref{s=1Theorem} may be regarded as a~version of Cartan's result (see also Corollary~12 in~\cite{Sluis}).
\end{Remark}

It turns out that these two theorems are consequences of the fact that the following matrices are inverses of each other, after taking into account the maps $\Phi^{(\infty)}$ and $\Psi^{(\infty)}$:
\begin{gather}\label{AB}
\bs{A} = \left(\begin{matrix}
A^0_0 &\cdots &A^0_p&\bs{0}&\cdots\\
A^1_0&\cdots&A^1_p&A_\infty&\cdots\\
\vdots&\vdots&\vdots&\vdots&\ddots
\end{matrix}
\right), \qquad
\bs{B} = \left(\begin{matrix}
B^0_0 &\cdots &B^0_q&\bs{0}&\cdots\\
B^1_0&\cdots&B^1_p&B_\infty&\cdots\\
\vdots&\vdots&\vdots&\vdots&\ddots
\end{matrix}
\right).
\end{gather}
In fact, suppose that $(M,\ct_M)$ and $(N,\ct_N)$ have $n_1$, $n_2$ states and $m_1$, $m_2$ controls, respectively. We can arrange that $n_1 = n_2=:n$ by partially prolonging one of the systems, if necessary. Moreover, since partial prolongations preserve the number of controls, this will have no effect on the values of $m_1$ or $m_2$. Then by Proposition~\ref{porq=0Prop}, there are two possible cases: $p = q = 0$ or $p,q>0$. In the former case, we have static equivalence. In the latter case, the form of matrix $\bs{A}$ implies that the $n+rm_2$ linearly independent components of $\bseta^0,\dots ,\bseta^{r}$ are all linear combinations of the $n+(r+p)m_1$ components of $\bsomega^0,\dots ,\bsomega^{r+p}$. When $m_2>m_1$, this is impossible because $n+(r+p)m_1< n+rm_2$ as long as
\begin{gather*}
r> \frac{m_1p}{m_2-m_1}.
\end{gather*}
Theorem \ref{sameSTheorem} follows. Furthermore, when $p,q>0$, $\bs{A}$ and $\bs{B}$ being inverses of each other requires that either $A_\infty B_\infty = \bs{0}$ or $B_\infty A_\infty =\bs{0}$. This is impossible when $m_1 = m_2 = 1$, in which case~$A_\infty$ and~$B_\infty$ are just nonvanishing functions. Theorem~\ref{s=1Theorem} is an immediate consequence.

More generally, we have the

\begin{Lemma}\label{rankLemma} Suppose that $\mathcal{E}$ is a dynamic equivalence of height $(p,q)$ $(p,q >0)$ between two control systems with $m$ controls and not necessarily the same number of states. The associated matrices $A_\infty$ and $B_\infty$ must satisfy
\begin{gather}\label{rankInequality1}
2\le \rank(A_\infty)+\rank(B_\infty) \le m.
\end{gather}
\end{Lemma}
\begin{proof} This is because $p,q >0$ implies $A_\infty B_\infty = \bs{0}$.\end{proof}

\subsection[The infinite permutation matrix $\mathcal{S}$]{The infinite permutation matrix $\boldsymbol{\mathcal{S}}$}

 Let $(M,\ct_M)$, $(N,\ct_N)$, $\bsomega^i$, $\bseta^i$, $\bs{A}$, $\bs{B}$ be as above. In~\cite{stackpole}, it is proved that there exist transformations
 \begin{gather} \label{ABtransformation}
\left(\begin{matrix}
\bar\bsomega^0\\
\bar\bsomega^1\\
\vdots
\end{matrix}
\right) = \left(\begin{matrix}
\bs{g}^0_0&\bs{0}&\dots\\
\bs{g}^1_0&\bs{g}^1_1&\dots\\
\vdots&\vdots&\ddots
\end{matrix}
\right)\left(\begin{matrix}
\bsomega^0\\
\bsomega^1\\
\vdots
\end{matrix}
\right), \qquad
\left(\begin{matrix}
\bar\bseta^0\\
\bar\bseta^1\\
\vdots
\end{matrix}
\right) = \left(\begin{matrix}
\bs{h}^0_0&\bs{0}&\dots\\
\bs{h}^1_0&\bs{h}^1_1&\dots\\
\vdots&\vdots&\ddots
\end{matrix}
\right)\left(\begin{matrix}
\bseta^0\\
\bseta^1\\
\vdots
\end{matrix}
\right),
\end{gather}
where $\bs{g}^i_i = \bs{g}^{i+1}_{i+1}$, $\bs{h}^i_i = \bs{h}^{i+1}_{i+1}$ for all $i\ge 1$, such that, pointwise,
\begin{enumerate}\itemsep=0pt
\item[C1.] $ \begin{cases}
\spn\big\{\bsomega^0,\dots ,\bsomega^k\big\} = \spn\big\{\bar\bsomega^0,\dots ,\bar\bsomega^k\big\},\\
\spn\big\{\bseta^0,\dots ,\bseta^k\big\} = \spn\big\{\bar\bseta^0,\dots ,\bar\bseta^k\big\},
\end{cases} k\ge 0$;
\item[C2.] $ \begin{cases}
\ed\bar\bsomega^\ell = -\bar\bsomega^{\ell+1}\W \ed t \mod \bar\bsomega^0,\dots ,\bar\bsomega^\ell,\\
\ed\bar\bseta^\ell = -\bar\bseta^{\ell+1}\W \ed t \mod \bar\bseta^0,\dots ,\bar\bseta^\ell,
\end{cases} \ell \ge 1$;
\item[C3.] ${\Phi^{(\infty)}}^*\bar\bseta = \bar{\bs{A}}{\bar\bsomega}$, ${\Psi^{(\infty)}}^*\bar\bsomega = \bar{\bs{B}}{\bar\bseta}$, where $\bar{\bs{A}}$, $\bar{\bs{B}}$ take the same form as~\eqref{AB} (in particular, all~$\bar{A}^k_{p+k}$ are equal for $k\ge 1$, etc.), and both $\bar{\bs{A}}$ and $\bar{\bs{B}}$ are \emph{infinite permutation matrices}, that is, each row/column of $\bar{\bs{A}}$ and $\bar{\bs{B}}$ contains a single $1$ with the rest of the entries being all $0$.
\end{enumerate}

In addition, we have

\begin{Proposition}\label{SProp} Assume that $p,q >0$. The matrices $\bar{\bs{A}}$ and $\bar{{\bs B}}$ satisfy $\bar{\bs{A}} = \bar{{\bs B}}^T.$ In particular,
 the infinite permutation matrix $\bar{\bs{A}}$ must be of the form
 \begin{gather} \label{Smatrix}
 \mathcal{S}:= \bar{\bs{A}} = \left(\begin{matrix}
\bar{A}^0_0 &\bar{A}^0_1&\cdots &\bar{A}^0_p&\bs{0}&\cdots\\
\bar{A}^1_0&\bar{A}^1_1&\cdots&\bar{A}^1_p&\bar{A}_\infty&\ddots\\
\vdots&\vdots&\ddots&\vdots&&\ddots\\
\bar{A}^q_0&\bar{A}^q_1&\cdots&\bar{A}^q_p\\
\bs{0}&\bar{B}_\infty^T&&\\
\vdots&\ddots&\ddots
\end{matrix}
\right) ;\end{gather}
in other words, $\bar{A}^{q+k}_k = \bar{B}_\infty^T$ $(k\ge 1)$, and $\bar{A}^{q+\ell}_k = \bs{0}$ for all $\ell >k$.
\end{Proposition}

\begin{proof} The fact $\bar{\bs{A}} = \bar{{\bs B}}^T$ follows from $\bar{\bs{A}}\bar{{\bs B}} = \bar{\bs{B}}\bar{\bs{A}} = \diag(1,1,\dots )$ and Property~C3 above. Consequently, $\bar{A}^{q+k}_k = \big(\bar{B}^k_{q+k}\big)^T =\bar{B}_\infty^T $ for $k\ge 1$. For a similar reason, $\bar{A}^{q+\ell}_k = \bs{0}$ for all $\ell >k$.\end{proof}

\begin{Definition}\label{rankMatrixDef} Let $\mathcal{S}$, taking the form of~\eqref{Smatrix}, be an infinite permutation matrix obtained from a dynamic equivalence with height $(p,q)$ $(p,q>0)$ between two control systems. Let $r^i_j = \rank\big(\bar{A}^i_j\big)$. We define the \emph{rank matrix} associated to~$\mathcal{S}$ to be
\begin{gather*}
\mathcal{R}(\mathcal{S}): = \big(r^i_j\big).
\end{gather*}
\end{Definition}

Given a dynamic equivalence, an associated matrix $\mathcal{S}$ may depend on the choice of the transformations $\bar\bsomega = \bs{G}\bsomega$ and $\bar\bseta = \bs{H}\bseta$.
However, we have

\begin{Proposition}\label{rankProp}
If $\mathcal{S}_1$ and $\mathcal{S}_2$ are two infinite permutation matrices obtained from the
same dynamic equivalence, then their rank matrices satisfy
\begin{gather*}
\mathcal{R}(\mathcal{S}_1) = \mathcal{R}(\mathcal{S}_2).
\end{gather*}
\end{Proposition}

\begin{proof}Suppose that the underlying dynamic equivalence has height $(p,q)$ $(p,q >0)$. One can write $\mathcal{S}_1$ and $\mathcal{S}_2$ in block forms
\begin{gather*}
\mathcal{S}_1 = \left(\begin{matrix}
U^0_0& \cdots & U^0_p & \bs{0}&\cdots\\
\vdots & \ddots & \vdots&&\ddots \\
U^q_0&\cdots & U^q_p \\
\bs{0}&&&\ddots\\
\vdots&\ddots
\end{matrix}
\right), \qquad \mathcal{S}_2 = \left(\begin{matrix}
V^0_0& \cdots & V^0_p & \bs{0}&\cdots\\
\vdots & \ddots & \vdots&&\ddots \\
V^q_0&\cdots & V^q_p \\
\bs{0}&&&\ddots\\
\vdots&\ddots
\end{matrix}
\right).
\end{gather*}
Let $u^i_j := \rank\big(U^i_j\big)$ and $v^i_j := \rank\big(V^i_j\big)$. Since $\mathcal{S}_1$, $\mathcal{S}_2$ arise from the same dynamic equivalence, there exist invertible block lower triangular matrices\footnote{Respectively, the block sizes of $\bs{K}$ and $\bs{L}$ are the same as those of $\bs{G}$ and $\bs{H}$.}
\begin{gather*}
\bs{K} = \left(\begin{matrix}
\bs{k}^0_0&\bs{0}&\dots\\
\bs{k}^1_0&\bs{k}^1_1&\dots\\
\vdots&\vdots&\ddots
\end{matrix}
\right), \qquad
\bs{L} = \left(\begin{matrix}
\bs{\ell}^0_0&\bs{0}&\dots\\
\bs{\ell}^1_0&\bs{\ell}^1_1&\dots\\
\vdots&\vdots&\ddots
\end{matrix}
\right),
\end{gather*}
where $\bs{k}^i_i = \bs{k}^{i+1}_{i+1}$, $\bs{\ell}^i_i = \bs{\ell}^{i+1}_{i+1}$ for all $i\ge 1$, such that
\begin{gather*}
\mathcal{S}_1\bs{K} = \bs{L}\mathcal{S}_2 = \left(\begin{matrix}
W^0_0& \cdots & W^0_p & \bs{0}&\cdots&\cdots&\cdots&\cdots\\
W^1_0& \cdots & \cdots& W^1_{p+1}&\ddots\\
\vdots & & &&\ddots&\ddots \\
W^q_0&\cdots & \cdots&\cdots&\cdots&W^q_{p+q}&\bs{0}&\cdots \\
\vdots&&&&&&\ddots&\ddots
\end{matrix}
\right).
\end{gather*}
As results of the forms of $\bs{K}$ and $\bs{L}$, we have
\begin{enumerate}\itemsep=0pt
\item[(i)] $u^i_{p+i} = v^i_{p+i}$ for all $i\ge 0$. This is because $W^i_{p+i} = U^i_{p+i}\bs{k}^{p+i}_{p+i} = \bs{\ell}^i_i V^i_{p+i}$, where both $\bs{k}^{p+i}_{p+i}$ and~$\bs{\ell}^i_i$ are invertible.
\item[(ii)] $u^0_{i} = v^0_i$ for all $0\le i<p$. To see why this is true, consider the submatrix $\big(W^0_i W^0_{i+1}\cdots W^0_p\big)$. For each $i<p$, its row rank equals to $v^0_i+v^0_{i+1}+\cdots+ v^0_p$, which must be equal to its column rank $u^0_i+ u^0_{i+1}+\cdots + u^0_p$.
\item[(iii)] $u^1_i = v^1_i$ for all $0\le i<p+1$. To see this, consider the submatrix
$\left(\begin{smallmatrix} W^0_i &\cdots& W^0_p &\bs{0} \\
W^1_i&\cdots &W^1_p&W^1_{p+1}
 \end{smallmatrix}\right)$. Its row rank $\big(v^0_i+\cdots+v^0_p\big)+\big(v^1_i+\cdots + v^1_{p+1}\big)$ must be equal to its column rank $\big(u^0_i+\cdots+u^0_p\big)+\big(u^1_i+\cdots +u^1_{p+1}\big)$. Let $i$ decrease from $p$ and use~$(ii)$. The desired result follows.

\item[(iv)] $u^j_i = v^j_i$ for $j\ge 2$, $0\le i< p+j$. This can be verified by a similar comparison between the column and row ranks of the submatrices
\begin{gather*}
\left(\begin{matrix} W^0_i &\cdots& W^0_p &\bs{0} &\cdots&\bs{0}\\
W^1_i&\cdots &W^1_p&W^1_{p+1}&\ddots&\vdots\\
\vdots&&&&\ddots&\bs{0}\\
W^j_i&\cdots &W^j_p&W^j_{p+1}&\cdots&W^j_{p+j}
 \end{matrix}\right), \qquad i<p+j.\end{gather*}
\end{enumerate}
This completes the proof.\end{proof}

As a consequence of Proposition~\ref{rankProp}, we have

\begin{Corollary}If $\mathcal{S}_1$ and $\mathcal{S}_2$ are two infinite permutation matrices obtained from a~dynamic equivalence with height $(p,q)$ $(p,q>0)$, then there exist block diagonal matrices
\begin{gather*}
\bs{K} = \diag\big(\bs{k}^i_i\big)_{i\ge 0}, \qquad \bs{L} = \diag\big(\bs{\ell}^i_i\big)_{i\ge 0},
\end{gather*}
where $\bs{k}^i_i$ and $\bs{\ell}^i_i$ are usual permutation matrices of appropriate sizes,\footnote{That is, the sizes of $\bs{k}^i_i$ and $\bs{\ell}^i_i$ are consistent with the product of block matrices $\bs{L}\mathcal{S}_2\bs{K}$.} such that
\begin{gather*}
\mathcal{S}_1 = \bs{L}\mathcal{S}_2 \bs{K}.
\end{gather*}
\end{Corollary}

\begin{proof} This is because $\mathcal{S}_1$ and $\mathcal{S}_2$ (i) are permutation matrices; and (ii) have the same rank in each pair of corresponding blocks.\end{proof}

\begin{Corollary}Let $\mathcal{S}$ be an infinite permutation matrix obtained from a dynamic equivalence of height $(p,q)$ $(p,q>0)$. Using the notations in~\eqref{AB}, the associated rank matrix $\mathcal{R}(\mathcal{S}) = \big(r^i_j\big)$ satisfies
\begin{gather*}
r^k_{p+k} =\rank(A_\infty), \qquad r^{q+k}_k = \rank(B_\infty), \qquad k\ge 0.
\end{gather*}
\end{Corollary}

\begin{proof} To see why this is true, first notice that a transformation~\eqref{ABtransformation} preserves the ranks of~$A_\infty$ and $B_\infty$; then use Property~P3.\end{proof}

\begin{Remark}A dynamic equivalence may be viewed as an invertible differential operator~\cite{ABMP94,Chetverikov17,Levine11}. In particular, one may, according to~\cite{Chetverikov17}, define the associated \emph{$d$-scheme of squares}. It is interesting to observe how the relations (3)--(5) in~\cite{Chetverikov17} resemble the conditions that $r^i_j$ satisfy for our rank matrix, but we will not pursue this relation further in the current article.
\end{Remark}

\section{The height of a dynamic equivalence}

Given two control systems $(M,\ct_M)$ ($n_1$ states, $m$ controls) and $(N,\ct_N)$ ($n_2$ states, $m$ controls) that are dynamically equivalent, it is interesting to ask: \emph{What are the possible heights of a~dynamic equivalence?} A particular instance is Theorem~\ref{s=1Theorem}, which tells us that the height suggests how control systems with $m = 1$ and $m>1$ are qualitatively different. The current section will present some new results in this direction.

\subsection{Some rank equalities and inequalities}

Throughout this section, let $(M,\ct_M)$ and $(N,\ct_N)$ be as above. Suppose that a dynamic equivalence between them has height $(p,q)$ with $p,q >0$. Let $\mathcal{S}$ be an associated infinite permutation matrix (equation~\eqref{Smatrix}), obtained from a choice of coframes $\big(\bar\bsomega^0,\bar\bsomega^1,\dots \big)$ and $\big(\bar\bseta^0,\bar\bseta^1,\dots \big)$. Let $\mathcal{R}(\mathcal{S}) = \big(r^i_j\big)$ be the corresponding rank matrix (Definition~\ref{rankMatrixDef}).

\begin{Proposition}\label{rankEquality} $r^i_j$ satisfy the following equalities
\begin{gather}\label{rankEqualities}
 \sum_{i\ge 0} r^i_0 = n_1, \qquad \sum_{j\ge 0}r^0_j = n_2, \qquad \sum_{i\ge 0} r^i_k = \sum_{j\ge 0} r^k_j = m, \qquad k = 1,2,\dots.
\end{gather}
\end{Proposition}
\begin{proof} This is because the matrix $\mathcal{S}$ is an infinite permutation matrix.\end{proof}

\begin{Proposition}\label{rankInequality2} $r^i_j$ satisfy the following inequalities:
\begin{enumerate}\itemsep=0pt
\item[$(i)$] for $i,j \ne 0$, then
\begin{gather*}
r^i_j \le \min\left\{\sum_{k = 0}^{j+1} r^{i+1}_k, \sum_{k = 0}^{i+1} r^k_{j+1}\right\};
\end{gather*}

\item[$(ii)$] for $j\ne 0$,
\begin{gather*}
r^0_j\le \min\left\{ r^0_{j+1}+ r^1_{j+1}, (n_2 - m)+\sum_{k = 0}^{j+1} r^1_k \right\};
\end{gather*}

\item[$(iii)$] for $i\ne 0$,
\begin{gather*}
r^i_0\le \min\left\{ r^{i+1}_0+ r^{i+1}_1, (n_1 - m)+\sum_{k = 0}^{i+1} r^k_1 \right\};
\end{gather*}

\item[$(iv)$] $r^0_0\le\min\big\{ (n_1-m)+r^0_1+r^1_1, (n_2-m)+r^1_0+ r^1_1\big\}$.
\end{enumerate}
\end{Proposition}
\begin{proof}To prove (i), the case when $r^i_j = 0$ is trivial. Otherwise, suppose that the submatrix~$\bar A^i_j$ of~$\mathcal{S}$ has $1$'s precisely at positions $(a_k, b_k)$, $1\le a_k, b_k\le m$, $1\le k \le r^i_j$. Dropping pullback symbols, we have
\begin{gather*}
\bar \eta^i_{a_{k}} = \bar\omega^j_{b_k}.
\end{gather*}
Condition~C2 demands
\begin{gather}\label{deta}
\ed\bar \eta^i_{a_{k}} \equiv - \bar \eta^{i+1}_{a_{k}}\W \ed t \mod \bseta^0,\dots ,\bseta^i
\end{gather}
and
\begin{gather}\label{domega}
\ed\bar \omega^j_{b_{k}} \equiv - \bar \omega^{j+1}_{b_{k}}\W \ed t \mod \bsomega^0,\dots ,\bsomega^j.
\end{gather}
At most $\sum\limits_{k = 0}^{j} r^{i+1}_k$ congruences in \eqref{deta} are reduced to the following form once the congruence is taken modulo $\bseta^0,\dots ,\bseta^i,\bsomega^0,\dots ,\bsomega^j$:
\begin{gather*}
\ed\bar \eta^i_{a_{k}} \equiv 0\mod\bseta^0,\dots ,\bseta^i,\bsomega^0,\dots ,\bsomega^j;
\end{gather*}
similarly, at most $\sum\limits_{k = 0}^{i} r^k_{j+1}$ congruences in \eqref{domega} are reduced to the following form once the congruence is taken modulo $\bsomega^0,\dots ,\bsomega^j, \bseta^0,\dots ,\bseta^i$:
\begin{gather*}
\ed\bar \omega^j_{b_{k}} \equiv 0 \mod\bsomega^0,\dots ,\bsomega^j, \bseta^0,\dots ,\bseta^i.
\end{gather*}
The remaining congruences in~\eqref{deta} and~\eqref{domega}, reduced modulo $\bseta^0,\dots ,\bseta^i,\bsomega^0,\dots ,\bsomega^j$, must match up as identical congruences; in particular, the corresponding $\bar\eta^{i+1}_{a_k}$ and $\bar\omega^{j+1}_{b_k}$ must be equal. Since the equalities $\bar\eta^{i+1}_{a_k}=\bar\omega^{j+1}_{b_k}$ are at most
$r^{i+1}_{j+1}$ in number, we obtain the inequalities
\begin{gather*}
r^i_j \le \sum_{k = 0}^{j+1}r^{i+1}_k, \qquad r^i_j\le \sum_{k = 0}^{i+1} r^i_{j+1},
\end{gather*}
which justifies~(i).

To prove (iv), suppose that $r^0_0>0$ and that the submatrix $\bar A^0_0$ has $1$'s precisely at positions $(a_k,b_k)$, $1\le a_k\le n_2$, $1\le b_k \le n_1$, $1\le k\le r^0_0$. We have
\begin{gather*}
\bar \eta^0_{a_k} = \bar\omega^0_{b_k}.
\end{gather*}
There exist functions $C^1_k,\dots ,C^s_k$, $D^1_k,\dots ,D^s_k$ such that
\begin{gather}\label{deta0}
\ed\bar \eta^0_{a_k} \equiv - \big(C^1_k \bar\eta^{1}_1+\cdots +C^s_k \bar\eta^{1}_s\big)\W \ed t \mod \bseta^0,\\
\label{domega0}
\ed\bar \omega^0_{b_k} \equiv - \big(D^1_k \bar\omega^{1}_1+\cdots + D^s_k \bar\omega^{1}_s\big)\W \ed t \mod \bsomega^0.
\end{gather}
Since the rank of $\ed\bseta^0$ (modulo $\bseta^0$) is $m$, we have
\begin{gather}\label{Cineq}
\rank\big(C^\alpha_k\big)\ge m - \big(n_2 - r^0_0\big);
\end{gather}
similarly, we have
\begin{gather}\label{Dineq}
\rank\big(D^\alpha_k\big)\ge m - \big(n_1 - r^0_0\big).
\end{gather}
On the other hand, by taking the congruences~\eqref{deta0} and~\eqref{domega0} reducing modulo both~$\bsomega^0$ and~$\bseta^0$, it is not hard to see that
\begin{gather}\label{CDineq}
\rank\big(C^\alpha_k\big) \le r^1_0+r^1_1, \qquad \rank\big(D^\alpha_k\big)\le r^0_1+r^1_1.
\end{gather}
Combining \eqref{Cineq}, \eqref{Dineq} and \eqref{CDineq}, we obtain the inequalities in~(iv).

The proofs of (ii) and (iii) are similar; we leave them to the reader.\end{proof}

\subsection{Admissible heights}

\begin{Theorem}\label{heightTheorem}
Let $\mathcal{E}$ be a dynamic equivalence between two control systems with $n_1$, $n_2$ states, respectively, and $m$ controls. The height $(p,q)$ of $\mathcal{E}$ must satisfy: if $p,q >0$, then
\begin{gather}\label{heightIneq}
\min\{(p - 1)\delta+r_1 p + n_1, (q-1)\delta+r_2 q + n_2\}\ge \max\{ r_1 p + n_1, r_2q + n_2\},
\end{gather}
where $r_1 = \rank(A_\infty)$, $r_2 = \rank(B_\infty)$, $\delta= m - r_1 - r_2$.
\end{Theorem}

\begin{proof} The rank matrix associated to $\mathcal{E}$ takes the form
\begin{gather*}
\mathcal{R} = \left(\begin{matrix}
r^0_0& \cdots & r^0_p & {0}&\cdots\\
\vdots & \ddots & \vdots&&\ddots \\
r^q_0&\cdots & r^q_p \\
{0}&&&\ddots\\
\vdots&\ddots
\end{matrix}
\right),
\end{gather*}
where $r^k_{p+k} = \rank(A_\infty) = r_1$, $r^{q+k}_k = \rank(B_\infty) = r_2$ for all $k\ge 0$.

Let
\begin{gather*}
C:= \sum_{i = 0}^{q-1}\sum_{j = p}^{p+i} r^i_j - \sum_{i = 0}^{q-1}r^i_{p+i},\qquad
D:= \sum_{j = 0}^{p-1}\sum_{i = q}^{q+j} r^i_j - \sum_{j = 0}^{p-1}r^{q+j}_j.
\end{gather*}

By Proposition \ref{rankEquality}, we have
\begin{gather}\label{CDCond}
0\le C \le (q-1)\delta, \qquad 0\le D\le (p-1)\delta.
\end{gather}
Furthermore, we have
\begin{gather*}
E:= (q-1)(m-r_1) - \big(n_1 - r^0_0- r_2\big) - C = (p-1)(m - r_2) - \big(n_2-r^0_0 - r_1\big) - D.
\end{gather*}
Using \eqref{CDCond}, we obtain
\begin{gather*}
 \max\{r_1p + n_1, r_2q + n_2\}\le E - r^0_0 + n_1+n_2\\
\qquad{} \le \min\{(p-1)\delta+r_1p+n_1, (q-1)\delta+r_2q +n_2\}.
\end{gather*}
The conclusion follows.\end{proof}

\begin{Corollary}\label{heightCor} Let $\mathcal{E}$ be a dynamic equivalence between two control systems with $n_1$, $n_2$ states, respectively, and $m$ controls. If $\rank(A_\infty)+\rank(B_\infty)=m$, then the height $(p,q)$ of $\mathcal{E}$ must satisfy: when $p,q>0$, we have
\begin{gather}\label{heightTheoremEq}
n_1+ \rank(A_\infty)\cdot p= n_2 + \rank(B_\infty)\cdot q.
\end{gather}
\end{Corollary}
\begin{proof} This is an immediate consequence of setting $\delta = 0$ in~\eqref{heightIneq}.\end{proof}

\begin{Theorem}\label{n=2Cor}The height $(p,q)$ of a dynamic equivalence between two control systems with $n_1$ and $n_2$ states, respectively, and $2$ controls must satisfy $n_1+p = n_2+q$.
\end{Theorem}

\begin{proof}Let $(M,\ct_M)$ and $(N,\ct_N)$ be the two control systems in question with a dynamic equivalence given by submersions $\Phi\colon M^{(p)}\rightarrow N$ and $\Psi\colon N^{(q)}\rightarrow M$.

First consider the case when $n_1 = n_2$. By Proposition \ref{porq=0Prop}, we have either $p = q = 0$ or $p,q>0$. In the former case, there is nothing to prove. In the latter case, when $m = 2$, Lemma~\ref{rankLemma} implies that the only possibility is $\rank(A_\infty) = \rank(B_\infty) = 1$; then apply Corollary~\ref{heightCor}.

It remains to consider the case when $n_1\ne n_2$. Assume that $n_1>n_2$. The case when $p,q>0$ follows from a similar argument as the above. Thus, it suffices to consider the case when $p = 0$. Suppose that, under a choice of coordinates $(t,\y,\v)$ on $N$, $\Psi$ depends nontrivially on $v_\alpha^{(q)}$ for some $\alpha$. Fixing an index $\beta\ne \alpha$, we can construct a partial prolongation $\big(\bar N,\bar\ct_N\big)$ by letting $\bar N = N\times\R^{n_1 - n_2}$ with coordinates $\big(t,\y,\v,v_\beta^{(1)},\dots ,v_\beta^{(n_1 - n_2)}\big)$ and $\bar\ct_N$ being generated by
\begin{gather*}
\ct_N, \,\ed v_\beta - v_\beta^{(1)}\ed t,\, \dots,\, \ed v_{\beta}^{(n_1 - n_2 - 1)} - v_\beta^{(n_1 - n_2)}\ed t.
\end{gather*}
Let $\bar\Phi\colon M^{(n_1 - n_2)}\rightarrow \bar N$ (indicated by the dashed arrow in the diagram below) denote the submersion induced by~$\Phi$. Let $\bar\pi^{(q)}\colon \bar N^{(q)}\rightarrow N^{(q)}$ be the submersion induced by the natural projection $\bar\pi\colon \bar N\rightarrow N$
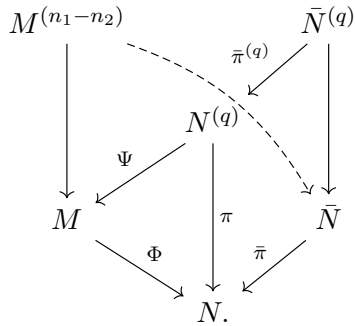
\begin{equation*}
\begin{tikzcd}[column sep=small]
M^{(n_1- n_2)} \arrow[rrdd, bend left=20, dashed]\arrow[dd] &&\bar N^{(q)} \arrow[ld,swap, "\bar\pi^{(q)}"]\arrow[dd]\\
& N^{(q)}\arrow[ld,swap, "\Psi"]\arrow[dd,"\pi"]&\\
M\arrow[rd,"\Phi"]&&\bar N\arrow[ld, swap,"\bar\pi"]\\
&N.
 \end{tikzcd}
\end{equation*}

It is easy to see that the pair of submersions $\bar\Phi$ and $\Psi\circ\bar\pi^{(q)}$ establishes a dynamic equivalence between $(M,\ct_M)$ and $\big(\bar N,\bar\ct_N\big)$ of height $(n_1 - n_2, q)$. Since $M$ and $\bar N$ have the same dimension, we conclude that $n_1 - n_2 = q$. This completes the proof.
\end{proof}

\begin{Remark} Theorem~\ref{n=2Cor} suggests a qualitative distinction between control systems with $m = 2$ and those with $m>2$. In fact, when $m>2$, it may be, for a dynamic equivalence, that $\rank(A_\infty)+\rank(B_\infty)< m$, and~\eqref{heightTheoremEq} does not need to hold. For details, see Example~\ref{example1} below.
\end{Remark}

\subsection{Examples}

\begin{Example}\label{example1} Let $(M,\ct_M)$ be a control system with 4 states and 3 controls, where $M$ has coordinates $(t, \x, \u)$ and $\ct_M$ is generated by
\begin{gather}\label{Example1CM}
\ed x_1 - u_1 \ed t, \qquad \ed x_2 - x_1 \ed t, \qquad \ed x_3 - u_2 \ed t, \qquad \ed x_4 - u_3 \ed t.
\end{gather}
A partial prolongation $(M,\ct_M)$ can be obtained by adjoining equations of the form $\dot u_\alpha^{(k)} = u_\alpha^{(k+1)}$ to the original system. For example, consider $N_1 = M\times \R^3$ with the coordinates
\begin{gather*}\big(t,\big(\x, u_1, u_2, u_2^{(1)}\big), \big(u_1^{(1)}, u_2^{(2)}, u_3\big)\big),\end{gather*}
where $u_1^{(1)}$, $u_2^{(1)}$ and $u_2^{(2)}$ are coordinates on the $\R^3$ component; let $\ct_1$ be the Pfaffian system generated by \eqref{Example1CM} and the $1$-forms
\begin{gather*}
\ed u_1 - u_1^{(1)}\ed t, \qquad \ed u_2 - u_2^{(1)}\ed t, \qquad \ed u_2^{(1)} - u_2^{(2)}\ed t.
\end{gather*}
Alternatively, consider $N_2 = M\times \R^3$ with the coordinates
\begin{gather*}\big(t,\big(\x, u_3, u_3^{(1)}, u_3^{(2)}\big), \big(u_1, u_2, u_3^{(3)}\big)\big),\end{gather*}
where $u_3^{(1)}$, $u_3^{(2)}$ and $u_3^{(3)}$ are coordinates on the~$\R^3$ component; let $\ct_2$ be the Pfaffian system generated by~\eqref{Example1CM} and the $1$-forms
\begin{gather*}
\ed u_3 - u_3^{(1)}\ed t, \qquad \ed u_{3}^{(1)} - u_3^{(2)} \ed t, \qquad \ed u_{3}^{(2)} - u_3^{(3)} \ed t.
\end{gather*}
The standard submersions $\Phi\colon N_1^{(3)}\rightarrow N_2$, $\Psi\colon N_2^{(2)}\rightarrow N_1$ give rise to a dynamic equivalence between $(N_1,\ct_1)$ and $(N_1,\ct_2)$ (both having~$7$ states and $3$ controls) with height $(p,q) = (3,2)$. The associated rank matrix is
\begin{gather*}
(r^i_j)=\left(
\begin{matrix}
4 &1 &1 &1&&&&&\\
2 &0 &0&0&1&&&&\\
1& 1 &0&0&0&1&&&\\
&1&1&0&0&0&1&&\\
&&1&1&0&0&0&1&\\
\hphantom{\ddots}&\hphantom{\ddots}&\hphantom{\ddots}&{\ddots}&\hphantom{\ddots}&\hphantom{\ddots}&\hphantom{\ddots}&\hphantom{\ddots}&\ddots
\end{matrix}
\right).
\end{gather*}
 In this example, \eqref{heightIneq} becomes an equality.
\end{Example}

\begin{Example} Let $(M,\ct_M)$ ($\dot\x = \f(\x,\u)$) and $(N,\ct_N)$ $(\dot \y = \f(\y,\v))$ be two copies of a same control system with $3$ states and $2$ controls:
\begin{gather*}
\begin{cases}
\dot x_1 = u_1,\\
\dot x_2 = u_2,\\
\dot x_3 = f(x_2,x_3,u_2),
\end{cases}\qquad
\begin{cases}
\dot y_1 = v_1,\\
\dot y_2 = v_2,\\
\dot y_3 = f(y_2,y_3,v_2).
\end{cases}
\end{gather*}
For any $p>1$, the following pair of submersions $\Phi\colon M^{(p)}\rightarrow N$ and $\Psi\colon N^{(p)}\rightarrow M$ define a~dynamic equivalence with height $(p,p)$ between $(M,\ct_M)$ and $(N,\ct_N)$:
\begin{gather*}
(\y,\v) = \Phi\big(\x,\u,\dots ,\u^{(p)}\big) = \big(u_2^{(p-1)} - x_1, x_2, x_3; u_2^{(p)} - u_1, u_2\big),\\
(\x,\u) = \Psi\big(\y,\v,\dots ,\v^{(p)}\big) = \big(v_2^{(p-1)} - y_1, y_2,y_3; v_2^{(p)} - v_1, v_2\big).
\end{gather*}
This shows that there is no \emph{a priori} upper bound of dynamic equivalences relating a fixed pair of control systems; it would only be meaningful to ask whether a `minimum
height' exists among all possible dynamic equivalences between such a fixed pair.
\end{Example}

\begin{Example} Assume $n_1>n_2$ in Theorem~\ref{n=2Cor}. This theorem tells us that even though $\dim N^{(q)} \ge \dim M$ (in order for $N^{(q)}$ to submerse onto~$M$) as long as
\begin{gather*}
 q \ge \frac{n_1 - n_2}{2},
\end{gather*}
an actual dynamic equivalence could only exist with $q\ge n_1 - n_2$. One may compare this fact with Theorem~52 in~\cite{Sluis}.

As an example, we consider the PVTOL system (see~\cite{MMR03}).

Using the coordinates $(t, (x,z,\theta, \dot x, \dot z, \dot \theta), (u_1,u_2))$ on $M$, let $\ct_M$ be generated by the six $1$-forms:
\begin{gather*}
\ed x - \dot x \ed t,\\
\ed z - \dot z\ed t,\\
\ed\theta - \dot\theta \ed t,\\
\ed \dot x - (-u_1\sin\theta +\epsilon u_2\cos\theta)\ed t,\\
\ed \dot z - (u_1\cos\theta+\epsilon u_2\sin\theta - 1)\ed t,\\
\ed\dot\theta - u_2 \ed t,
\end{gather*}
where $\epsilon$ is a constant.

Using the coordinates $(t, (y_1, y_2), (\dot y_1,\dot y_2))$ on $N$, let $\ct_N$ be the trivial system generated by the two $1$-forms:
\begin{align*}
&\ed y_1 - \dot y_1 \ed t,\\
&\ed y_2 - \dot y_2 \ed t.
\end{align*}

By the argument above, if there exists a dynamic equivalence between $(M,\ct_M)$ and $(N,\ct_N)$ (aka.~$(M,\ct_M)$ being `differentially flat'), then such a dynamic equivalence
must satisfy $q\ge n_1 - n_2 = 6 - 2 = 4$.

In fact, one can verify that the relation
\begin{gather*}
(y_1, y_2) = (x - \epsilon \sin \theta, z + \epsilon \cos\theta)
\end{gather*}
induces a dynamic equivalence between $(M,\ct_M)$ and $(N,\ct_N)$ of height $(p,q) = (0,4)$. For more details, see \cite{MMR03}.
\end{Example}

\subsection*{Acknowledgements}
We would like to thank our referees for their careful reading of the manuscript and helpful comments that led to considerable improvement of this article. The first and third authors were supported in part by NSF grant DMS-1206272. The first author was supported in part by a~Collaboration Grant for Mathematicians from the Simons Foundation.

\pdfbookmark[1]{References}{ref}
\LastPageEnding

\end{document}